\documentclass[12pt]{amsart}

\usepackage{amsfonts, amsthm, amsmath}

\usepackage{rotating}

\usepackage{tikz}

\usepackage{graphics}

\usepackage{amssymb}

\usepackage{amscd}

\usepackage[latin2]{inputenc}

\usepackage{t1enc}

\usepackage[mathscr]{eucal}

\usepackage{indentfirst}

\usepackage{graphicx}

\usepackage{graphics}

\usepackage{pict2e}

\usepackage{mathrsfs}

\usepackage{enumerate}
\usepackage[pagebackref]{hyperref}
\hypersetup{backref, pagebackref, colorlinks=true}
\usepackage{cite}
\usepackage{color}
\usepackage{epic}
\usepackage{hyperref} 
\numberwithin{equation}{section}
\theoremstyle{plain}

\newtheorem{theorem}{Theorem}[section]

\newtheorem{conjecture}[theorem]{Conjecture}

\theoremstyle{definition}

\newtheorem{Def}[theorem]{Definition}

\newtheorem{remark}[theorem]{Remark}

\newtheorem{?}[theorem]{Problem}

\def\boxit#1{\leavevmode\hbox{\vrule\vtop{\vbox{\kern.33333pt\hrule
    \kern1pt\hbox{\kern1pt\vbox{#1}\kern1pt}}\kern1pt\hrule}\vrule}}

\usepackage{collectbox}



\begin{document}

\title[On a generalized crank]{On a generalized crank for $k$-colored partitions}

\author[S. Fu]{Shishuo Fu}
\address[Shishuo Fu]{College of Mathematics and Statistics, Chongqing University, Huxi Campus LD506, Chongqing 401331, P.R. China}
\email{fsshuo@cqu.edu.cn}

\author[D. Tang]{Dazhao Tang}

\address[Dazhao Tang]{College of Mathematics and Statistics, Chongqing University, Huxi Campus LD208, Chongqing 401331, P.R. China}
\email{dazhaotang@sina.com}

\date{\today}

\begin{abstract}
A generalized crank ($k$-crank) for $k$-colored partitions is introduced. Following the work of Andrews-Lewis and Ji-Zhao, we derive two results for this newly defined $k$-crank. Namely, we first obtain some inequalities between the $k$-crank counts $M_{k}(r,m,n)$ for $m=2,3$ and $4$, then we prove the positivity of symmetrized even $k$-crank moments weighted by the parity for $k=2$ and $3$. We conclude with several remarks on furthering the study initiated here.
\end{abstract}



\maketitle



\section{Introduction}\label{sec1}
A \emph{partition} \cite{Andr2} $\pi$ of a positive integer $n$ is a finite weakly decreasing sequence of positive integers $\pi_{1}\geq\pi_{2}\geq\cdots\geq\pi_{r}>0$ such that $\sum_{i=1}^{r}\pi_{i}=n$. The $\pi_{i}$ are called the \emph{parts} of the partition. In 1944, Dyson \cite{Dys} defined the \emph{rank} of a partition as the largest part minus the number of parts and then observed that the rank appears to give combinatorial interpretations for Ramanujan's two celebrated partition congruences modulo $5$ and $7$. Unfortunately, Dyson's rank fails to do the same thing for Ramanujan's third partition congruence modulo $11$. Nevertheless, Dyson \cite{Dys} postulated the existence of another partition statistic, which he coined as the \emph{``crank''}, that would similarly explain Ramanujan's third partition
congruence. In 1988, Andrews and Garvan \cite{A-G} finally captured Dyson's elusive crank of partitions, motivated by the crank of certain vector partitions, which was first studied by Garvan\cite{Gar1}. In an earlier paper \cite{FT} we considered, in the same vein, two families $\overline{r}$ and $r^*$ of multiranks for multipartitions, multi-overpartitions and multi-pods that lead to similar combinatorial interpretations for congruence properties enjoyed by these types of partitions. We introduce here yet a third family of generalized cranks for multipartitions, which apparently do not possess any significance in explaining congruences but, as we are going to demonstrate, have in common several neat distributional properties with the Andrews-Garvan-Dyson crank.

Our starting point is a recent paper by Bringmann and Dousse \cite{BD}, in which they settled a longstanding conjecture by Dyson \cite[p.~172]{Dys2} concerning the limiting shape of the crank generating function. More precisely, they considered a family of inverse theta functions defined for $k\in\mathbb{N}$ by
\begin{align}\label{gf:k colored part crank}
\mathcal{C}_{k}(z,q)=\sum_{n=0}^{\infty}\sum_{m=-\infty}^{\infty}M_{k}(m,n)z^{m}q^{n}:=\dfrac{(q;q)_{\infty}^{2-k}}{(zq;q)_{\infty}(z^{-1}q;q)_{\infty}}.
\end{align}

Here and in what follows, we adopt the following customary notations \cite{Andr2} in partitions and $q$-series:
\begin{align*}
(a;q)_{j} &:=\prod_{i=0}^{j-1}(1-aq^{i}), \quad j\in\mathbb{N}_0\cup\{\infty\}.
\end{align*}

Since $P(q):=(q;q)_{\infty}^{-1}$ is the generating function for the ordinary partitions, the product side of \eqref{gf:k colored part crank} suggests the following definition.
\begin{Def}\label{k-crank}
For $k\geq 2$, the generalized crank (abbreviated as $k$-crank in what follows) of a $k$-colored partition $\overrightarrow{\pi}=(\pi^{(1)},\pi^{(2)},\cdots,\pi^{(k)})$ is defined as
\begin{align}\label{generalized crank}
\textrm{$k$-crank}(\overrightarrow{\pi})=\ell(\pi^{(1)})-\ell(\pi^{(2)}),
\end{align}
where $\ell(\pi^{(i)})$ denotes the number of parts in $\pi^{(i)}$.
\end{Def}

\begin{remark}
We note that the special case $k=2$ yields the Hammond-Lewis birank \cite{HL} of $2$-colored partitions, and the case $k=3$ corresponds to the authors' multirank $r^{*}$ for $3$-colored partitions \cite{FT}. For $k>3$, the $k$-crank defined by \eqref{generalized crank} appears to be new.
\end{remark}

In view of \eqref{gf:k colored part crank}, Definition~\ref{k-crank} and the generating function for crank \eqref{gf:ordinary pa}, it is clear that for $k\geq 2$, $M_{k}(m,n)$ (resp.~$M_1(m,n)$) enumerates the number of $k$-colored partitions (resp.~ordinary partitions) of $n$ with $k$-crank (resp. crank) equals $m$. By convention, we simply write $M(m,n)$ for $M_1(m,n)$. Furthermore, we denote $M_{k}(r,m,n)$ (resp. $M(r,m,n)$) by the number of $k$-colored partitions (resp.~ordinary partitions) of $n$ with $k$-crank (resp. crank) congruent to $r$ modulo $m$.

Given a $k$-colored partition $\overrightarrow{\pi}=(\pi^{(1)},\pi^{(2)},\cdots,\pi^{(k)})$ with crank equals, say $m$, we can pair with it a $k$-colored partition with crank equals $-m$, simply by swapping $\pi^{(1)}$ and $\pi^{(2)}$. This observation immediately establishes the symmetry for $k$-crank:
\begin{align}\label{sym-kcrank}
M_k(m,n)=M_k(-m,n).
\end{align}
This parallels the symmetry for the crank of ordinary partitions, which is not so obvious due to its asymmetric definition (see \cite{BG} for a direct combinatorial proof).

The generating function for $M(m,n)$ was given in \cite{A-G,Gar1}:
\begin{align}\label{gf:ordinary pa}
\mathcal{C}_1(z,q)=\mathcal{C}(z,q) &=\sum_{n=0}^{\infty}\sum_{m=-\infty}^{\infty}M(m,n)z^{m}q^{n}=\dfrac{(q;q)_{\infty}}{(zq;q)_{\infty}(z^{-1}q;q)_{\infty}}.
\end{align}
Putting $z=-1$ in \eqref{gf:ordinary pa} gives
\begin{align*}
\sum_{n=0}^{\infty}\left(M(0,2,n)-M(1,2,n)\right)q^{n} &=\dfrac{(q;q)_{\infty}}{(-q;q)_{\infty}^{2}},
\end{align*}
whose coefficient $M(0,2,n)-M(1,2,n)$ alternates in sign, a fact that was first observed by Andrews and Lewis.
\begin{theorem}[Theorem~1 in \cite{AL}]\label{parity results}
For all $n\geq0$,
\begin{align*}
M(0,2,2n) &>M(1,2,2n),\\
M(1,2,2n+1) &>M(0,2,2n+1).
\end{align*}
\end{theorem}

Following the work of Andrews and Lewis, we study the $k$-crank of $k$-colored partitions modulo 2, 3, 4 and obtain comparable results. The following is the $m=2$ case.
\begin{theorem}\label{thm:alternate thm}
For $n\geq0$,
\begin{align}
M_{k}(0,2,2n) &>M_{k}(1,2,2n),\quad~\emph{for}~k=2,3,4,\label{analog ine1}\\
M_{k}(1,2,2n+1) &>M_{k}(0,2,2n+1),\quad~\emph{for}~k=2,3,\label{analog ine2}\\
M_{4}(0,2,2n+1) &= M_{4}(1,2,2n+1), \label{strict iden1}\\
M_{k}(0,2,n) &> M_{k}(1,2,n), \quad \emph{for}~k\geq5.\label{analog ine3}
\end{align}
\end{theorem}

The rest of the paper is organized as follows. In Section~\ref{sec:2}, we establish some inequalities between $k$-cranks of $k$-colored partitions modulo 2, 3, and 4. Next in Section~\ref{sec:3} we generalize \eqref{analog ine1} and \eqref{analog ine2} further by introducing the symmetrized $k$-crank moments . We conclude in the last section with some remarks and one conjecture on the unimodality of $M_k(m,n)$.

\section{$k$-crank modulo 2, 3, and 4}\label{sec:2}

\subsection{The case $m=2$}

\begin{proof}[Proof of Theorem \ref{thm:alternate thm}]
Taking $k=2, z=-1$ in \eqref{gf:k colored part crank} gives
\begin{align}
\sum_{n=0}^{\infty}\left(M_{2}(0,2,n)-M_{2}(1,2,n)\right)q^{n} &=\dfrac{1}{(-q;q)_{\infty}^{2}}:=f(q),\label{gf:f}
\end{align}
say, then we have to show that the coefficient of $q^{n}$ in \eqref{gf:f} is positive/negative according to whether $n$ is even or odd. In other words, we need to prove that the coefficients of $f(-q)$ are all positive. Since
\begin{align*}
f(-q) &=\dfrac{1}{(q;-q)_{\infty}^{2}}=\dfrac{(-q;q)_{\infty}^{2}}{(-q^{2};q^{2})_{\infty}^{2}}=(-q;q^{2})_{\infty}^{2},
\end{align*}
then the coefficients of $f(-q)$ are all positive. Actually $f(-q)$ is the generating function of pairs of partitions into distinct odd parts. This gives us the $k=2$ case of \eqref{analog ine1} and \eqref{analog ine2}.

Similariy, taking $k=3, z=-1$ in \eqref{gf:k colored part crank} leads to
\begin{align*}
\sum_{n=0}^{\infty}\left(M_{3}(0,2,n)-M_{3}(1,2,n)\right)q^{n} &=\dfrac{1}{(-q;q)_{\infty}^{2}(q;q)_{\infty}}=\dfrac{(q;q^{2})_{\infty}}{(q^{2};q^{2})_{\infty}},
\end{align*}
which analogously gives us the $k=3$ case of \eqref{analog ine1} and \eqref{analog ine2}.

For $k\geq4$, we have
\begin{align}\label{gf3:2-colored pa}
\sum_{n=0}^{\infty}\left(M_{k}(0,2,n)-M_{k}(1,2,n)\right)q^{n} &=\dfrac{1}{(-q;q)_{\infty}^{2}(q;q)_{\infty}^{k-2}}=\dfrac{1}{(q^{2};q^{2})_{\infty}^{2}(q;q)_{\infty}^{k-4}}.
\end{align}
If $k=4$, we get \eqref{strict iden1} and the $k=4$ case of \eqref{analog ine1}, since the power of $q$ in \eqref{gf3:2-colored pa} must be even. When $k\geq5$, \eqref{analog ine3} is obvious since \eqref{gf3:2-colored pa} contains the factor $1/(q;q)_{\infty}$. This completes the proof.
\end{proof}

\subsection{The case $m=3$}
In the same paper \cite{AL}, Andrews and Lewis proposed the following conjecture, which was first solved by Kane \cite{Kan} and reproved in a more systematic setting by Kim \cite{Kim} later.
\begin{theorem}\label{Kane thm}
For $n\geq0$,
\begin{align*}
M(0,3,3n) &>M(1,3,3n),\\
M(0,3,3n+1) &<M(1,3,3n+1),\\
M(0,3,3n+2) &<M(1,3,3n+2),\quad \emph{for }n\neq1,4,5,\\
M(0,3,3n+2) &=M(1,3,3n+2),\quad \emph{for }n=4,5.
\end{align*}
\end{theorem}

In contrast, we have the following result when considering $k$-crank modulo $3$.
\begin{theorem}\label{+--thm}
For all $n\geq0$,
\begin{align}
\label{ine:2mod3 3n}M_{2}(0,3,3n) &>M_{2}(1,3,3n),\\
\label{ine:2mod3 3n+1}M_{2}(0,3,3n+1) &<M_{2}(1,3,3n+1),\\
\label{ine:2mod3 3n+2}M_{2}(0,3,3n+2) &<M_{2}(1,3,3n+2),\\
\label{ine:3mod3 >}M_{3}(0,3,3n) &> M_{3}(1,3,3n),\\
\label{ine:3mod3 =}M_{3}(0,3,3n+1) &= M_{3}(1,3,3n+1),\; M_{3}(0,3,3n+2) = M_{3}(1,3,3n+2),\\
\label{ine:kmod3}M_{k}(0,3,n) &> M_{k}(1,3,n),  \quad \emph{for}~k\geq4.
\end{align}
\end{theorem}
\begin{proof}
We first note that assuming $k=2, z=e^{2\pi i/3}$ in \eqref{gf:k colored part crank} gives
\begin{align}
\sum_{n=0}^{\infty}\left(M_{2}(0,3,n)-M_{2}(1,3,n)\right)q^{n} &=\dfrac{(q;q)_{\infty}}{(q^{3};q^{3})_{\infty}}=(q;q^{3})_{\infty}(q^{2};q^{3})_{\infty}.\label{inf Bowern conj}
\end{align}
To prove \eqref{ine:2mod3 3n}, \eqref{ine:2mod3 3n+1} and \eqref{ine:2mod3 3n+2}, we only need to show that the signs of coefficients of $q^{3n}$, $q^{3n+1}$ and $q^{3n+2}$ in \eqref{inf Bowern conj} are ``$+--$''. To that end, we consider the following $3$-dissection:
\begin{align*}
\dfrac{(q;q)_{\infty}}{(q^{3};q^{3})_{\infty}}=\dfrac{J_{1}}{J_{3}}&=\dfrac{J_{12,27}}{J_{3}}-q\dfrac{J_{6,27}}{J_{3}}-q^{2}\dfrac{J_{3,27}}{J_{3}},
\end{align*}
where $J_{s}:=(q^{s};q^{s})_{\infty}$ and $J_{s,t}:=(q^{s};q^{t})_{\infty}(q^{t-s};q^{t})_{\infty}(q^{t};q^{t})_{\infty}$ for $1\leq s < t$.
This is a simple consequence of the Jacobi Triple Product identity, see for example \cite[p.~48, Entry~31]{Ber} for a proof. Next for $k\geq 3$, we also set $z=e^{2\pi i/3}$ in \eqref{gf:k colored part crank} to have
\begin{align}\label{gf:mod3 identity}
\sum_{n=0}^{\infty}(M_{k}(0,3,n)-M_{k}(1,3,n))q^{n} &=\dfrac{1}{(q^{3};q^{3})_{\infty}(q;q)_{\infty}^{k-3}}.
\end{align}

If $k=3$, we have \eqref{ine:3mod3 >} and \eqref{ine:3mod3 =} since the power of $q$ in \eqref{gf:mod3 identity} is always divisible by $3$. When $k\geq4$, note that \eqref{gf:mod3 identity} contains the factor $1/(q;q)_{\infty}$, thus we have \eqref{ine:kmod3}.
\end{proof}

\begin{remark}
Two remarks on Theorem~\ref{+--thm} are in order. First note that \eqref{ine:2mod3 3n}--\eqref{ine:2mod3 3n+2} can be viewed as the infinite version of the so-called ``First Borwein conjecture'' \cite[(1.1)]{Andr4}, and can be deduced from Theorem~2.1 in \cite{Andr4}. Secondly, Chan and Mao obtained an improvement \cite[Corollary 1.8]{CM} on Theorem \ref{Kane thm}. It is then natural to ask if there exists an analogous improvement on Theorem \ref{+--thm}.
\end{remark}

\subsection{The case $m=4$}
Andrews and Lewis obtained the following inequalities of crank modulo 4.
\begin{theorem}[Theorem~3 in \cite{AL}]
For $n>0$,
\begin{align*}
M(0,4,2n) &>M(1,4,2n), \quad~\emph{for}~n\neq1,\\
M(0,4,2n-1) &<M(1,4,2n-1), \quad~\emph{for}~n\neq2,\\
M(2,4,2n) &>M(1,4,2n),\\
M(2,4,2n-1) &<M(1,4,2n-1).
\end{align*}
\end{theorem}
Similarly, the numbers $M_{k}(r,4,n)$ satisfy the following relations.
\begin{theorem}
For $n>0$,
\begin{align}
M_{2}(0,4,4n) &>M_{2}(2,4,4n)>M_{2}(1,4,4n),\quad \emph{for}~n\neq1,\label{ine:2mod4 4n}\\
M_{2}(2,4,4n+2) &>M_{2}(0,4,4n+2)>M_{2}(1,4,4n+2),\label{ine:2mod4 4n+2}\\
M_{2}(1,4,2n+1) &>M_{2}(0,4,2n+1)=M_{2}(2,4,2n+1), \label{ine:2mod4 2n+1}\\
M_{3}(0,4,2n) &>M_{3}(1,4,2n)=M_{3}(2,4,2n),\nonumber\\
M_{3}(0,4,2n+1) &=M_{3}(1,4,2n+1)>M_{3}(2,4,2n+1),\nonumber\\
M_{k}(0,4,n) &>M_{k}(1,4,n)>M_{k}(2,4,n),\quad \emph{for}~k\geq4. \nonumber
\end{align}
\end{theorem}

\begin{proof}
Firstly, setting $k=2, z=i$ in in \eqref{gf:k colored part crank} gives
\begin{align}
\sum_{n=0}^{\infty}\left(M_{2}(0,4,n)-M_{2}(2,4,n)\right)q^{n} &=\dfrac{1}{(-q^{2};q^{2})_{\infty}}=(q^2;q^4)_{\infty},\label{gf:g2}
\end{align}
which does not have any odd powers of $q$ in the expansion, and the coefficients of $q^{4n}$ are all positive except for $n=1$, while the coefficients of $q^{4n+2}$ are all negative.

Recall from \eqref{gf:f},
\begin{align*}
 &\sum_{n=0}^{\infty}\left(M_{2}(0,4,n)+M_{2}(2,4,n)-2M_{2}(1,4,n)\right)q^{n}\\
 =&\sum_{n=0}^{\infty}\left(M_{2}(0,2,n)-M_{2}(1,2,n)\right)q^{n}=\dfrac{1}{(-q;q)_{\infty}^{2}},
\end{align*}
combining this with \eqref{gf:g2} we obtain
\begin{align*}
\sum_{n=0}^{\infty}\left(M_{2}(0,4,n)-M_{2}(1,4,n)\right)q^{n} &=\dfrac{1}{2}\left\{\dfrac{1}{(-q;q)_{\infty}^{2}}+\dfrac{1}{(-q^{2};q^{2})_{\infty}}\right\}:=\alpha(q),
\end{align*}
say. Now
\begin{align*}
\alpha(-q) &=\dfrac{1}{2}\left\{\dfrac{1}{(q;-q)_{\infty}^{2}}+\dfrac{1}{(-q^{2};q^{2})_{\infty}}\right\} =\dfrac{1}{2}\left\{(-q;q^{2})_{\infty}^{2}+(q^{2};q^{4})_{\infty}\right\}\\
 &=\dfrac{1}{2}(-q;q^{2})_{\infty}\left\{(-q;q^{2})_{\infty}+(q;q^{2})_{\infty}\right\}=(-q;q^{2})_{\infty}\sum_{n=0}^{\infty}
 \mathcal{OD}_{0}(n)q^{n},
\end{align*}
where $\mathcal{OD}_{0}(n)$ is the number of partitions of $n$ into an even number of distinct odd parts. Consequently the coefficients of $q^{n}$ in $\alpha(-q)$ are positive except for $n=2$.

In just the same way, we see that
\begin{align*}
\sum_{n=0}^{\infty}\left(M_{2}(2,4,n)-M_{2}(1,4,n)\right)q^{n}=\dfrac{1}{2}\left\{\dfrac{1}{(-q;q)_{\infty}^{2}}-\dfrac{1}{(-q^{2};q^{2})_{\infty}}\right\} &:=\beta(q),
\end{align*}
say, then
\begin{align*}
\beta(-q)=\dfrac{1}{2}(-q;q^{2})_{\infty}\left\{(-q;q^{2})_{\infty}-(q;q^{2})_{\infty}\right\}
=(-q;q^{2})_{\infty}\sum_{n=0}^{\infty}\mathcal{OD}_{1}(n)q^{n},
\end{align*}
where $\mathcal{OD}_{1}(n)$ is the number of partitions of $n$ into an odd number of distinct odd parts. Therefore the coefficients of $q^{n}$ in $\beta(-q)$ are positive for $n>0$. Thus we have proved \eqref{ine:2mod4 4n}--\eqref{ine:2mod4 2n+1}. The arguments for cases with $k\geq 3$ are similar so we choose to omit them.
\end{proof}

\section{2-crank and 3-crank moments: symmetrized and weighted}\label{sec:3}
In 2003, Atkin and Garvan \cite{AG} demonstrated the importance of the moments of ranks and cranks in the study of further partition congruences. Later, Andrews \cite{Andr3} considered a cominatorial interpretation of the moments of rank by introducing a symmetrized rank moment. This in turn motivated Garvan \cite{Gar2} to consider the symmetrized crank moments in his study of the higher order spt-functions. To be more precise, the $j^{\text{th}}$ symmetrized crank moment as defined by Garvan \cite{Gar2} is
\begin{align*}
\mu_{j}(n)=\sum_{m=-n}^{n}\left(m+\left\lfloor\frac{j-1}{2}\right\rfloor\atop j\right)M(m,n),
\end{align*}
where $\lfloor x \rfloor:=\max\{m\in\mathbb{Z}: m\leq x\}$ is the usual floor function. From the symmetry $M(m,n)=M(-m,n)$, it is clear that $\mu_{2j+1}(n)=0$.

In our setting, a natural analog for $k$-crank is given by
\begin{align*}
\mu_{j,k}(n)=\sum_{m=-n}^{n}\left(m+\left\lfloor\frac{j-1}{2}\right\rfloor\atop j\right)M_{k}(m,n),
\end{align*}
which is also meaningful for even $j$ only due to the symmetry \eqref{sym-kcrank} enjoyed by the $k$-cranks.

Recently, Ji and Zhao \cite{JZ} considered the $2j^{\text{th}}$ crank moment weighted by the parity of cranks, i.e.,
\begin{align*}
\mu_{2j}(-1,n):=\sum_{m=-n}^{n}\binom{m+j-1}{2j}(-1)^{m}M(m,n).
\end{align*}
They showed the following positivity property of $(-1)^{n}\mu_{2j}(-1,n)$, which encompasses Theorem~\ref{parity results} as the $j=0$ case.
\begin{theorem}[Theorem 1.2 in \cite{JZ}]\label{thm:J-Z}
For $n\geq j\geq0$, $(-1)^{n}\mu_{2j}(-1,n)>0$.
\end{theorem}

Motivated by the work of Ji and Zhao, we consider the $2j^{\text{th}}$ symmetrized moments of $k$-colored partitions weighted by the parity of $k$-cranks, defined as
\begin{align}
\mu_{2j,k}(-1,n) &:=\sum_{m=-n}^{n}\binom{m+j-1}{2j}(-1)^{m}M_{k}(m,n), \quad \textrm{for } k\geq 2.\label{mu2jk-1}
\end{align}

When $j=0$, $k=2, 3$, \eqref{mu2jk-1} reduces to
\begin{align*}
\mu_{0,2}(-1,n) &=M_{2}(0,2,n)-M_{2}(1,2,n),\\
\mu_{0,3}(-1,n) &=M_{3}(0,2,n)-M_{3}(1,2,n).
\end{align*}
So the following result parallels Theorem~\ref{thm:J-Z} and includes \eqref{analog ine1} and \eqref{analog ine2} as the $j=0$ case.
\begin{theorem}\label{thm:positivity}
For $n\geq j\geq0$, we have
\begin{align}
(-1)^{n}\mu_{2j,2}(-1,n) &>0, \label{2-color po}\\
(-1)^{n}\mu_{2j,3}(-1,n) &>0. \label{3-color po}
\end{align}
\end{theorem}

With the help of Andrew's $j$-fold generalization of $q$-Whipple's theorem \cite[p.199, Theorem~4]{Andr1}, we can derive the following explicit generating functions for $\mu_{2j,k}(-1,n)$.
\begin{align}\label{gf:mu2jk1}
&\sum_{n=0}^{\infty}\mu_{2j,k}(-1,n)q^{n} \\
=&\dfrac{1}{(q;q)_{\infty}^{k-2}(-q;q)_{\infty}^{2}}\sum_{n_{j}\geq n_{j-1}\geq\cdots\geq n_{1}\geq1}\dfrac{(-1)^{j}q^{n_{1}+n_{2}+\cdots+n_{j}}}{(1+q^{n_{1}})^{2}(1+q^{n_{2}})^{2}\cdots(1+q^{n_{j}})^{2}}.\nonumber
\end{align}

Furthermore, the above generating functions \eqref{gf:mu2jk1} is equivalent to the following form, which is the key identity for the proof of Theorem~\ref{thm:positivity}.
\begin{align}\label{gf:mu2jk2}
&\sum_{n=0}^{\infty}\mu_{2j,k}(-1,n)q^{n} \\
=&\dfrac{1}{(q;q)_{\infty}^{k-2}(-q;q)_{\infty}^{2}}\sum_{m_{j}> m_{j-1}>\cdots>m_{1}\geq1}\dfrac{(-1)^{m_{j}}m_{1}(m_{2}-m_{1})\cdots(m_{j}-m_{j-1})q^{m_{j}}}{(1-q^{m_{1}})(1-q^{m_{2}})\cdots(1-q^{m_{j}})}.\nonumber
\end{align}

The proofs of \eqref{gf:mu2jk1} and \eqref{gf:mu2jk2} are similar to those of the corresponding results in \cite{JZ}, thus we omit the details here.

Now we are ready to derive \eqref{2-color po} and \eqref{3-color po}.
\begin{proof}[Proof of Theorem~\ref{thm:positivity}]
Replacing $q$ by $-q$ in \eqref{gf:mu2jk2} and putting $k=2$, we see that
\begin{align*}
 &\sum_{n=0}^{\infty}(-1)^{n}\mu_{2j,2}(-1,n)q^{n}\\
 =&\dfrac{1}{(q;-q)_{\infty}^{2}}\sum_{m_{j}> m_{j-1}>\cdots>m_{1}\geq1}\dfrac{m_{1}(m_{2}-m_{1})\cdots(m_{j}-m_{j-1})q^{m_{j}}}{(1-(-q)^{m_{1}})(1-(-q)^{m_{2}})\cdots(1-(-q)^{m_{j}})}\\
 =&(-q;q^{2})_{\infty}\sum_{m_{j}> m_{j-1}>\cdots>m_{1}\geq1}\dfrac{m_{1}(m_{2}-m_{1})\cdots(m_{j}-m_{j-1})q^{m_{j}}(-q;q^{2})_{\infty}}{(1-(-q)^{m_{1}})(1-(-q)^{m_{2}})
 \cdots(1-(-q)^{m_{j}})}.
\end{align*}
Given $m_{j}> m_{j-1}>\cdots>m_{1}\geq1$, define
\begin{align}\label{odd power}
\sum_{m=0}^{\infty}g_{m_{1},m_{2},\cdots,m_{j}}(m)q^{m}:=\dfrac{(-q;q^{2})_{\infty}}{(1-(-q)^{m_{1}})(1-(-q)^{m_{2}})\cdots(1-(-q)^{m_{j}})}.
\end{align}
For each $m_{i}$, we discuss by two cases according to the parity: (i) $m_{i}$ is odd, then such factors $(1-(-q)^{m_{i}})$ in the denominator will all be cancelled by $(-q;q^{2})_{\infty}$ since the $m_{i}$ are distinct; (ii) $m_{i}$ is even, then the coefficients of $1/(1-(-q)^{m_{i}})$ are all nonnegative. Thus we arrive at $g_{m_{1},m_{2},\cdots,m_{k}}(m)\geq 0$ and $g_{m_{1},m_{2},\cdots,m_{k}}(0)=1$, together with the factor $q^{m_j}, m_j\geq j$, we can deduce \eqref{2-color po}.

Similarly, replacing $q$ by $-q$ in \eqref{gf:mu2jk2} and taking $k=3$, we find that
\begin{align*}
 &\sum_{n=0}^{\infty}(-1)^{n}\mu_{2j,3}(-1,n)q^{n}\\
 =&\dfrac{1}{(-q;-q)_{\infty}(q;-q)_{\infty}^{2}}\sum_{m_{j}> m_{j-1}>\cdots>m_{1}\geq1}\dfrac{m_{1}(m_{2}-m_{1})\cdots(m_{j}-m_{j-1})q^{m_{j}}}{(1-(-q)^{m_{1}})(1-(-q)^{m_{2}})\cdots(1-(-q)^{m_{j}})}\\
 =&\dfrac{1}{(q^{2};q^{2})_{\infty}}\sum_{m_{j}> m_{j-1}>\cdots>m_{1}\geq1}\dfrac{m_{1}(m_{2}-m_{1})\cdots(m_{j}-m_{j-1})q^{m_{j}}(-q;q^{2})_{\infty}}{(1-(-q)^{m_{1}})(1-(-q)^{m_{2}})\cdots(1-(-q)^{m_{j}})}.
\end{align*}
The rest of the proof is similar by analysing \eqref{odd power} and is omitted.
\end{proof}

\section{Final Remarks}\label{sec4}
We conclude with several questions that merit further investigation.
\begin{enumerate}[1)]
\item Atkin and Garvan \cite{AG} defined the $j^{\text{th}}$ moment of the crank by
\begin{align*}
M_{j}(n) &=\sum_{m=-n}^{n}m^{j}M(m,n),
\end{align*}
then they reproved the following beautiful identity
\begin{align}\label{Dyson eq}
\sum_{m=-n}^{n}m^{2}M(m,n)=2np(n),
\end{align}
due to Dyson \cite{Dys2}, who gave a combinatorial proof.

Following the same line as the proof of Atkin and Garvan and by applying the chain rule for taking derivative, we also obtain an analogue of \eqref{Dyson eq} for $k$-colored partitions.
For all positive integer $k$, we have
\begin{align}\label{analog Dyson thm}
\sum_{m=-n}^{n}m^{2}M_{k}(m,n)=\dfrac{2}{k}np_{k}(n),
\end{align}
where $p_{k}(n)$ denotes the number of $k$-colored partitions of $n$. This is equivalent to saying the mean-square $k$-crank of $k$-colored partitions of $n$ is exactly $2n/k$. It should be clear that the right hand side of \eqref{analog Dyson thm} is always an integer, and it would be appealing to find a combinatorial proof for this identity.

\item Lewis \cite{Lew} showed that
\begin{align*}
N(0,2,2n) &<N(1,2,2n),\quad \textrm{if}~n\neq1,\\
N(1,2,2n+1) &<N(0,2,2n+1),\quad \textrm{if}~n\neq0,
\end{align*}
where $N(r,m,n)$ denotes the number of partitions of $n$ with rank congruent to $r$ modulo $m$. His proof is combinatorial (bijective) in nature and consists of the construction of maps
\begin{align*}
\{\textrm{partitions~of~$2n$~of~even~rank}\} &\rightarrow\{\textrm{partitions~of~$2n$~of~odd~rank}\}, \text{ and }\\
\{\textrm{partitions~of~$2n+1$~of~odd~rank}\} &\rightarrow\{\textrm{partitions~of~$2n+1$~of~even~rank}\},
\end{align*}
that are injective, but not surjective. One naturally wonders if analogous combinatorial analysis can be applied to prove Theorem~\ref{thm:alternate thm}, or in general any one of the inequalities presented in Section~\ref{sec:2}.

\item When $k\geq4$, $(-1)^{n}\mu_{2j,k}(-1,n)$ do not possess the positivity property as $(-1)^{n}\mu_{2j,2}(-1,n)$ and $(-1)^{n}\mu_{2j,3}(-1,n)$. However, they may have other interesting properties. On the other hand, there are many more rank and crank identities. For example,
\begin{align*}
M(0,8,4n+1)+M(1,8,4n+1)=M(3,8,4n+1)+M(4,8,4n+1),
\end{align*}
and others for moduli 5, 7, 8, 9, 10 and 11. The readers are referred to \cite{Gar1,Gar3,LS} and the references therein for more details. It would be interesting to find similar identities for $k$-crank.

\item A sequence of numbers $a_{1}, a_{2}, \cdots, a_{n}$ is \emph{unimodal} if it never increases after the first time it decreases, i.e., if for some index $j$ we have $a_{1}\leq a_{2}\leq\cdots a_{j-1}\leq a_{j}\geq a_{j+1}\geq\cdots\geq a_n$. The numerical evidence (see Table~\ref{Tab1}) suggests the following conjecture.
\begin{conjecture}\label{conj:unimodal}
For $n\geq 0$ and $k\geq 2$, the sequence $\{M_{k}(m,n)\}_{m=-n}^{n}$ is unimodal except for $n=1,k=2$.
\end{conjecture}
\begin{table}[tbp]\caption{A Table of values of $M_{2}(m,n)$}\label{Tab1}
\centering
\begin{tabular}{cccccccccccccc}
\hline
$n\setminus m$ &0 &1 &2 &3 &4 &5 &6 &~7 &~8 &~9 &10 &11 & 12\\
\hline
0 &1 & & & & & & & & & & &\\
1 &0 &1 & & & & & & & & & &\\
2 &1 &1 &1 & & & & & & & & &\\
3 &2 &2 &1 &1 & & & & & & & &\\
4 &4 &3 &3 &1 &1 & & & & & & &\\
5 &6 &6 &4 &3 &1 &1 & & & & & &\\
6 &11 &9 &8 &5 &3 &1 &1 & & & & &\\
7 &16 &16 &12 &9 &5 &3 &1 &1 & & & &\\
8 &27 &24 &21 &14 &10 &5 &3 &1 &1 & & &\\
9 &40 &39 &31 &25 &15 &10 &5 &3 &1 &1 & & &\\
10 &63 &59 &51 &37 &27 &15 &10 &5 &3 &1 &1 & &\\
11 &92 &90 &75 &60 &41 &28 &16 &10 &5 &3 &1 &1 &\\
12 &141 &131 &116 &90 &67 &43 &29 &16 &10 &5 &3 &1 &1 \\
\hline
\end{tabular}
\end{table}

Thanks to the symmetry \eqref{sym-kcrank}, we only need to prove
\begin{align}\label{monotone}
M_k(m,n)\geq M_k(m+1,n) \text{ for } 0\leq m\leq n.
\end{align}
Note that for $k\geq 3$, and fix $0\leq m \leq n$, we have $$M_{k}(m,n)=\sum_{j=m}^{n}M_{k-1}(m,j)p(n-j),$$ so it will suffice to prove this unimodality for $k=2$. Moreover, using two easily described maps we can prove that for $0\leq m \leq n$,
\begin{align*}
M_2(m,n)&\leq M_2(m,n+1),\\
M_2(i,i+m)&<M_2(i+1,i+m+1), \text{ for } 0\leq i \leq m-1,\\
M_2(m,2m)&=M_2(m+i,2m+i), \text{ for } i \geq 0.
\end{align*}
Consequently we establish ``half'' of \eqref{monotone}, i.e.,
$$M_k(m,n)\geq M_k(m+1,n) \text{ for } \lfloor n/2 \rfloor\leq m\leq n.$$

Unlike other properties shared by both crank and $k$-crank, this unimodality is not true for crank. For example,
\begin{align*}
M(n,n)=M(n-2,n)=1\quad~\textrm{and}~M(n-1,n)=0,\; \text{for all }n\geq 4.
\end{align*}

Lastly, we note that the asymptotic formula obtained by Bringmann and Manschot \cite[Corollary~1.3]{BM} makes it plausible to give a computer-aided proof of Conjecture~\ref{conj:unimodal}, but it would still be interesting to seek for a combinatorial proof.

\end{enumerate}

\section*{Acknowledgement}
Both authors were supported by the Fundamental Research Funds for the Central Universities (No.~CQDXWL-2014-Z004) and the National Natural Science Foundation of China (No.~11501061).


\begin{thebibliography}{99}

\bibitem{Andr1}G.~E.~Andrews, \emph{Problems and prospects for basic hypergeometric functions}, In: Askey, R. (ed.) Theory and Applications of Special Functions, pp. 191--224. Academic Press, New York (1975)

\bibitem{Andr2}G.E. Andrews, \emph{The Theory of Partitions}, Encyclopedia of Mathematics and Its Applications, Vol. 2 (G.-C. Rota, ed.), Addison-Wesley, Reading, 1976
    (Reprinted: Cambridge Univ. Press, London and New York, 1984).

\bibitem{Andr3}G.~E.~Andrews, \emph{Partitions, Durfee symbols, and the Atkin--Garvan moments of ranks}, Invent. math. {\bf 169} (2007), 37--73.

\bibitem{Andr4}G.~E.~Andrews, \emph{On a conjecture of Peter Borwein}, J. Symbolic Comput. {\bf 20} (1995): 487--501.

\bibitem{A-G}G. E. Andrews, F.G. Garvan, \emph{Dyson's crank of a partition}, Bull. Amer. Math. Soc. {\bf 18} (1988): 167--171.

\bibitem{AL}G. E. Andrews, R. Lewis, \emph{The ranks and cranks of partitions moduli 2, 3 and 4}, J. Number Theory. {\bf 85} (2000): 74--84.

\bibitem{AG}A. O. L. Atkin and F. G. Garvan, \emph{Relations between the ranks and cranks of partitions}, Ramanujan. J. {\bf 7} (2003): 343--366.

\bibitem{Ber} B.~C.~Berndt, \emph{Ramanujan's Notebooks, Part III} (Springer--Verlag, New York, 1991).

\bibitem{BG} A.~Berkovich and F.~G.~Garvan, \emph{Some observations on Dyson's new symmetries of partitions}, J. Combin. Theory Ser. A. {\bf 100} (2002): 61--93.

\bibitem{BD}K. Bringmann and J. Dousse, \emph{On Dyson's conjecture and the uniform asymptotic behavior of certain inverse theta functions}, Trans. Amer. Math. Soc. {\bf 368}(5) (2015): 3141--3155.

\bibitem{BM}K. Bringmann and J. Manschot, \emph{Asymptotic formulas for coefficients of inverse theta functions}, arXiv (2013). (\href{http://arXiv.org/abs/1304.7208}{arXiv:1304.7208}).

\bibitem{CM}S.~Chan and R.~Mao, \emph{The rank and crank of partitions modulo 3}, Int. J. Number Theory. \textbf{12}(4) (2016): 1027--1053.

\bibitem{Dys}F. J. Dyson, \emph{Some guesses in the theory of partitions}, Eureka (Reprinted: Selected Papers, Am. Math. Soc. Providence, pp. 51--56 (1996)). \textbf{8} (1944), 10--15.

\bibitem{Dys2}F. J. Dyson, \emph{Mappings and symmetries of partitions}, J. Combin. Theory Ser. A. \textbf{51} (1989): 169--180.

\bibitem{FT}S.~Fu and D.~Tang, \emph{Multiranks and classical theta functions}, arXiv (2016). (\href{http://arXiv.org/abs/1612.01141}{arXiv:1612.01141}).

\bibitem{Gar1}F. G. Garvan, \emph{New combinatorial interpretations of Ramanujan's partition congruences mod 5, 7, and 11}, Trans. Amer. Math. Soc. \textbf{305} (1988): 47--77.

\bibitem{Gar3}F. G. Garvan, \emph{The crank of partitions mod 8, 9 and 10}, Trans. Amer. Math. Soc. \textbf{322} (1990): 79--94.

\bibitem{Gar2}F. G. Garvan, \emph{Higher order spt-functions}, Adv. Math. \textbf{228} (2011): 241--265.

\bibitem{HL}P. Hammond and R. Lewis, \emph{Congruences in ordered pairs of partitions}, Int. J. Math. Math. Sci. 2004, nos. 45--48, 2509--2512.

\bibitem{JZ}Kathy Q.~Ji and Alice~X.~H.~Zhao, \emph{The crank moments weighted by the parity of cranks}, Ramanujan. J. to appear.

\bibitem{Kan}D.~M.~Kane, \emph{Resolution of a conjecture of Andrews and Lewis involving cranks of partitions}, Proc. Amer. Math. Soc. \textbf{132}(8) (2004): 2247--2256.

\bibitem{Kim}B.~Kim, \emph{Periodicity of signs of Fourier coefficients of eta quotients},  J. Math. Anal. Appl. \textbf{385} (2012): 998--1004.

\bibitem{Lew}R. Lewis, \emph{The ranks of partitions modulo 2}, Discrete. Math. \textbf{167/168} (1997): 445--449.

\bibitem{LS}R. Lewis and N. Santa-Gadea, \emph{On the rank and crank modulo 4 and 8}, Trans. Amer. Math. Soc. \textbf{341} (1994): 449--465.


\end{thebibliography}
\end{document}